\documentclass[psamsfonts]{amsart}

\newtheorem{thm}{Theorem}[section]
\newtheorem{lemma}[thm]{Lemma}
\newtheorem{dfn}[thm]{Definition}
\newtheorem{cor}[thm]{Corollary}
\newtheorem{prop}[thm]{Proposition}

\newcommand{\set}[2]{\{\,\,#1\,\,|\,\,#2\,\}}
\newcommand{\sete}[1]{\{\,#1\,\}}
\newcommand{\cp}{\mathbb CP}
\newcommand{\rp}{\mathbb RP}
\newcommand{\z}{\mathbb Z}
\newcommand{\Ker}{\mathrm{Ker}}
\newcommand{\Ima}{\mathrm{Im}}
\newcommand{\strips}[2]{\widetilde{\mathcal M}_J(L_0,L_1:#1,#2)}
\newcommand{\astrips}{\widetilde{\mathcal M}_J(L_0,L_1)}
\newcommand{\ccstrips}[2]{\overline{\mathcal M}_J(L_0,L_1:#1,#2)}
\newcommand{\ccastrips}{\overline{\mathcal M}_J(L_0,L_1)}
\newcommand{\cstrips}[2]{{\mathcal M}_J(L_0,L_1:#1,#2)}
\newcommand{\castrips}{{\mathcal M}_J(L_0,L_1)}

\newcommand{\ccastripss}[3]{\overline{\mathcal M}_{#1}(#2,#3)}
\newcommand{\nov}{\Lambda_{\Gamma,\omega}}
\newcommand{\unov}{\Lambda_{\z_2}}

\begin{document}

\title[Floer homology of $(\rp^{2n-1},T^{2n-1})$]{Lagrangian Floer homology of the Clifford torus and real projective space in odd dimensions}
\author{Garrett Alston}
\date{November 22, 2008}
\address{Department of Mathematics, University of Wisconsin, Madison}
\email{alstom@math.wisc.edu}


\maketitle

\section{Introduction}
The Clifford torus and real projective space are both Lagrangian submanifolds of complex projective space.
The Clifford torus is 
$$T^k=\set{[z_0:\cdots:z_k]\in\cp^k}{|z_0|=|z_1|=\cdots=|z_k|},$$
and real projective space is
$$\rp^k=\set{[z_0:\cdots:z_k]\in\cp^k}{z_i\in\mathbb R}.$$
They intersect in the $2^k$ points $[\pm1:\cdots:\pm1]$.
Two interesting questions to ask are the following: Can they be disjoined from each other by Hamiltonian isotopy?
If not, what is the minimum number of points that they must intersect in?
In this article we use Lagrangian Floer homology to investigate these questions.
The main result is
\begin{thm}
$$HF(\rp^{2n-1},T^{2n-1}:\z_2)= (\z_2)^{2^n}.$$
\end{thm}
The Floer chain group is generated by the intersection points of the submanifolds and the homology is invariant under Hamiltonian isotopy.
Therefore Theorem 1.1 immediately implies
\begin{thm}
$\rp^{2n-1}$ and $T^{2n-1}$ always intersect in at least $2^{n}$ points under Hamiltonian isotopy.
\end{thm}

Theorem 1.1 is true only in odd dimensions because the Floer homology is not defined in even dimensions due to disk bubbling.
There is also bubbling in odd dimensions, but the disk contributions cancel each other out.
It is actually known that $T^k$ and $\rp^k$ cannot be displaced from each other for any $k$.
This has been proved, using different methods, in various other papers.
See \cite{ep}, \cite{bc1}, \cite{bc2}, and \cite{t}.

In \cite{oh3} Oh conjectures that $T^k$ is volume minimizing in its Hamiltonian deformation class.
Due to a Lagrangian Crofton's formula (see \cite{van},\cite{oh3}), this problem is related to intersections of $\phi(T^k)$ with $\xi\cdot\rp^k$, for $\phi$ a Hamiltonian diffeomorphism and $\xi\in U(k+1)$.
Theorem 1.2 does not answer the conjecture, but we do get
\begin{cor}
For any Hamiltonian diffeomorphism $\phi$, we have 
$$\frac{\textrm{vol}(\phi(T^{2n-1}))}{\textrm{vol}(\rp^{2n-1})}\geq \frac{2^n}{ 2n}.$$
\end{cor}

We briefly describe the organization of this article.
Section 2 contains an overview of Lagrangian Floer homology.
We clearly indicate the hypotheses that need to be checked for the Floer homology to be well-defined.
The verification of some of the hypotheses is relegated to Section 5.
Readers familiar with Floer homology may want to skip to Section 3, where we determine all the Floer trajectories and write a formula for the boundary operator.
Section 4 contains the computation of the homology.
In Section 6 we carry out the same calculation for coefficients in the Novikov ring $\unov$.
Finally in Section 7 we briefly discuss Corollary 1.3.

I would like to thank Yong-Geun Oh for many helpful discussions on this topic.
Much of the exposition in this article is based upon his papers and a series of classes he taught on Floer homology.
Also, it was through him and the unpublished notes of A. Ivanshina that I learned how to do the dimension 3 case.

\section{Lagrangian Floer Homology}
We briefly describe how Floer homology is constructed.
See \cite{fl1}, \cite{fl2}, \cite{fl3}, \cite{fl4}, \cite{oh1}, and \cite{rs} for the details.
We will use only $\z_2$ coefficients until Section 6, so until then $\z_2$ will be dropped from the notation.

Let $(M,\omega)$ be a compact symplectic manifold, $L_0$ and $L_1$ two closed Lagrangian submanifolds that intersect transversely, and $J$ a time-dependent almost complex structure compatible with $\omega$.
The Floer chain group $CF(L_0,L_1)$ is the $\z_2$ vector space generated (formally) by $L_0\cap L_1$.
A Floer trajectory---or ($J$-)holomorphic strip---is a map
$$u:\mathbb R\times[0,1]\to M$$
that satisfies
\begin{equation}\label{eq:dbar}
\left\{
\begin{array}{l}
\bar\partial_J u={\partial\over\partial s}u+J_t(u){\partial \over \partial t} u=0, \\
u(\cdot,0)\in L_0,\,u(\cdot,1)\in L_1, \\
u(-\infty,\cdot)\in L_0\cap L_1, u(+\infty,\cdot)\in L_0\cap L_1. \\
\end{array}
\right.
\end{equation}
($\mathbb R\times[0,1]$ is to be viewed as a subset of $\mathbb C$ with coordinates $s+it$.)
Solutions of $\bar\partial_Ju=0$ with top and bottom Lagrangian boundary conditions satisfy the final asymptotic condition in (\ref{eq:dbar}) if and only if the energy of $u$
$$E(u)={1\over 2}\int_{[0,1]\times\mathbb R}|\bar\partial_J u|^2$$
is finite.
The space of all holomorphic strips that run from $p\in L_0\cap L_1$ to $q\in L_0\cap L_1$ will be denoted $\strips{p}{q}$.
Let
$$\astrips=\bigcup_{p,q\in L_0\cap L_1}\strips{p}{q}.$$
If the linearization $D_u\bar\partial_J$ of $\bar\partial_J$ is surjective for every $u\in\astrips$ then each $\strips{p}{q}$ is a smooth manifold (different components may have different dimensions).
Let $\mathcal J^{reg}$ denote the set of all such $J$.
$\mathcal J^{reg}$ is a set of the second category, and from now on we assume $J\in\mathcal J^{reg}$.
If $u\in\strips{p}{q}$ then 
$$\dim(T_u\strips{p}{q})=\mathrm{Index}(D_u\bar\partial_J).$$
The index of $D_u\bar\partial_J$ is equal to the spectral flow of $\bar\partial_J$ along $u$, and this in turn is an invariant of the homotopy class of $u$ and is equal to $\mu(u)$, the Maslov index of $u$.

If $u\in\astrips$ then $u(\cdot+s_0,\cdot)$ is also in $\astrips$ for any $s_0$.
In other words, there is a natural $\mathbb R$ action on the space of holomorphic strips.
We define more spaces by modding out by the $\mathbb R$ action:
$$\cstrips{p}{q}=\strips{p}{q}/\mathbb R,$$
$$\castrips=\astrips/\mathbb R.$$
An isolated trajectories is trajectory $u$ such that the equivalence class $[u]$ is a $0$-dimensional component of $\castrips$.
When no confusion can arise we will not distinguish between $u$ and $[u]$.
Let $n(p,q)$ be the mod-$2$ number of isolated trajectories in $\strips{p}{q}$.
The boundary operator
$$\partial:CF(L_0,L_1)\to CF(L_0,L_1)$$
is defined by
$$\partial(p)=\sum_qn(p,q)\cdot q.$$

Under certain topological conditions on $M$, $L_0$, and $L_1$, Floer proved that $\partial^2=0$.
Therefore, the Floer homology group
$$HF(L_0,L_1)=\Ker(\partial)/\Ima(\partial)$$
is defined.
Moreover, he showed that $HF(L_0,L_1)$ does not depend upon the choice of $J\in\mathcal J^{reg}$, and also that
$$HF(\Phi_0(L_0),\Phi_1(L_1))=HF(L_0,L_1)$$
for any Hamiltonian diffeomorphisms $\Phi_0$ and $\Phi_1$.
The reason $\partial^2=0$ is that the moduli spaces $\castrips$ can be compactified by adding on broken trajectories.
$\partial^2$ counts the number of boundary components of the $1$-dimensional part of the compactified moduli space, and hence must be zero mod-$2$.
The compactified moduli spaces will be denoted $\ccstrips{p}{q}$ and $\ccastrips$.

The topological restrictions imposed by Floer do not hold in our present case ($L_0=\rp^k$, $L_1=T^k$).
However, the results of Oh in \cite{oh1} imply that the Floer homology is still defined if $k$ is odd.
We cite the relevant theorems after some more definitions.

Let $L$ be a compact Lagrangian submanifold.
Two homomorphisms
$$I_\mu:\pi_2(M,L)\to \mathbb Z,$$
$$I_\omega:\pi_2(M,L)\to \mathbb R$$
are defined as follows:
For each map $w:(D^2,\partial D^2)\to (M,L)$, $I_\mu(w)$ is defined to be the Maslov number of the bundle pair $(w^*TM,(w|\partial D^2)^*TL)$.
$I_\omega$ is defined by $I_\omega(w)=\int_{D^2} w^*\omega$.
$L$ is called monotone if there exists a constant $c>0$ such that $I_\omega=c\cdot I_\mu$.
The minimal Maslov number $\Sigma_L$ is defined to be the positive generator of $\Ima(I_\mu)\subset\mathbb Z$.

\begin{thm}[\cite{oh1} Theorems 4.4, 5.1]
Assume that $L_0$ and $L_1$ are monotone Lagrangian submanifolds.
Assume further that $\Sigma_{L_i}\geq 3$ for $i=0,1$ and $\mathrm{Im}(\pi_1(L_i))\subset \pi_1(M)$ is a torsion subgroup for at least one of $i=0,1$. 
Then there exists a dense subset $\mathcal J^\prime\subset \mathcal J^{reg}$ such that if $J\in\mathcal J^\prime$ then
\begin{enumerate}
\item $\partial$ is well-defined, 
\item $\partial^2=0$, and
\item $HF(L_0,L_1)$ is independent of $J$ and Hamiltonian isotopy.
\end{enumerate}
\end{thm}

Statements 1 and 2 are equivalent to compactness of the 0- and 1-dimensional parts of $\ccastrips$, respectively.
The proofs of the theorems in \cite{oh1} show that $\mathcal J^\prime$ is the set of $J\in\mathcal J^{reg}$ for which the compactness holds.
Actually, because only the 0- and 1-dimensional components of $\castrips$ are used to define the Floer homology, it is not necessary for $J$ to be in $\mathcal J^{reg}$, but only that $D_u\bar\partial_J$ be surjective for those $u$ with Maslov index 1 or 2.

We state some well-know facts about $\rp^k$ and $T^k$; the proofs can be found in \cite{oh1}, \cite{oh2}, and \cite{cho1}.
$\rp^k$ and $T^k$ are both monotone Lagrangian submanifolds.
The minimal Maslov number of $\rp^k$ is $k+1$.
The minimal Maslov number of $T^k$ is 2 for all $k$.
$\rp^1$ is Hamiltonian isotopic to $T^1$, and $HF(\rp^1,\rp^1)=(\z_2)^2$.
Hence we may assume $k\geq 3$, and so the minimal Maslov number of $\rp^k$ is at least 4.

Theorem 2.1, therefore, does not directly apply because $T^k$ has Maslov index 2 disks.
To analyze the problem caused by Maslov index 2 disks we need the following theorem about Gromov-Floer compactness (see for example \cite{fl1}).
\begin{thm}
Let $u_n\in\cstrips{p}{q}$ be a sequence of holomorphic strips with constant Maslov index $I$ and energy $E(u_n)$ bounded above.
Then there exists a subsequence converging to the cusp-curve $u_\infty=(\underline u,\underline v,\underline w)$, where $\underline u=(u_1',\ldots,u_i')$ is an $(i-1)$-broken trajectory connecting $p$ to $q$, $\underline v=(v_1,\ldots,v_j)$ is a collection of holomorphic sphere bubbles, and $\underline w=(w_1,\ldots,w_k)$ is a collection of holomorphic disk bubbles with boundary lying entirely on $L_0$ or entirely on $L_1$.
Furthermore,
\begin{equation}\label{eq:conv}
I=\sum \mathrm\mu(u_\alpha')+2\sum c_1(v_\beta^*TM)+\sum I_\mu(w_\gamma).
\end{equation}
$c_1$ denotes the first Chern class and $\mu$ the Maslov index.
\end{thm}
Note that we abused notation and said $u_n\in\castrips$.
Technically, $u_n\in\astrips$, and the convergence only holds up to $\mathbb R$ translation.
That is, assuming we have already replaced the original sequence with a convergent subsequence, the convergence means that there exists sequences $s_n^\alpha\in\mathbb R$, $1\leq\alpha\leq i$, such that $u_n(\cdot+s_n^\alpha,\cdot)$ converges to $u_\alpha'$.
The convergence is $C^\infty$-uniform on every compact subset $K\subset\mathbb R\times[0,1]$ that does not contain any points where disk and sphere bubbles attach to $u_\alpha'$.

Now suppose $u_n\in\cstrips{p}{q}$ is a sequence with $\mu(u_n)=1$ for all $n$.
Under the assumptions of Theorem 2.1 it is true that all $u_n$ have the same energy (\cite{oh1} Proposition 2.7).
Theorem 2.2 thus applies and gives a convergent subsequence.
Monotonicity and regularity imply all the terms on the right-hand side of (\ref{eq:conv}) are positive.
The assumption $\Sigma_{L_i}\geq 3$ then implies that $u_\infty$ consists of a single trajectory.
That is, a subsequence of $u_n$ converges to an element of $\cstrips{p}{q}$.
Hence the 0-dimensional part of $\castrips$ is compact---so $\partial$ is well-defined.
The same argument works if it is only true that $\Sigma_{L_i}\geq 2$.

Consider now compactness of the 1-dimensional part of $\castrips$.
Let $u_n\in\cstrips{p}{q}$ be a sequnce with $\mu(u_n)=2$.
If $\Sigma_{L_i}\geq 3$ then the same argument works.
More precisely, the 1-dimensional part of $\castrips$ is compact up to splittings into 1-broken trajectories, with each 1-broken trajectory consisting of two isolated trajectories.
Therefore the boundary components that are added to the 1-dimensional part of the compactified moduli space $\ccastrips$ consist entirely of 1-broken trajectories---it follows that $\partial^2=0$.

If we only assume $\Sigma_{L_i}\geq 2$ then the situtation is more complicated.
If $p\neq q$ a similar argument works.
If $p=q$ two things can happen:
\begin{enumerate}
\item $\underline u=\emptyset$, $\underline v=(v_1)$ with $c_1(v_1)=1$, and $\underline w=\emptyset$,
\item $\underline u=\emptyset$, $\underline v=\emptyset$, and $\underline w=(w_1)$ with $\mu(w_1)=2$.
\end{enumerate}

In case 1, a sphere bubble $v_1$ appears at the point $p$.
A dimension counting argument shows that for a dense subset of $\mathcal J^{reg}$ the moduli space of holomorphic spheres with one marked point misses the 0-dimensional submanifold $L_0\cap L_1$ under the evaluation map.
Hence for these $J$ case 1 cannot occur.

Case 2 cannot be avoided.
In the addendum to \cite{oh1} it is noted that each holomorphic disk $w:(D^2,\partial D^2)\to(M,L_i)$ can be viewed as a boundary component of some 1-dimensional component of $\castrips$, as long as the linearization $D_w\bar\partial_J$ is surjective and the evaluation map on the moduli space of disks with one marked point is transverse to $L_0\cap L_1$.
In particular, $\partial^2=0$ no longer follows from the compactness of the 1-dimensional part of $\ccastrips$, because the boundary components contain elements besides broken trajectories.
However, we do have
$$<\partial^2(p),p>=\Phi_0(p)+\Phi_1(p),$$ 
where $\Phi_i(p)$ is the number mod-$2$ of Maslov index 2 disks with boundary on $L_i$ that pass through $p$.
Therefore, if $\Phi_0(p)+\Phi_1(p)= 0$ then $\partial^2=0$.

In summary, $HF(\rp^k,T^k)$ is well-defined, and $J$ can be used to calculate it, if and only if the following items hold:
\begin{enumerate}
\item No $J$-holomorphic spheres with $c_1=1$ pass through $L_0\cap L_1$.
\item The evaluation map on the moduli space of $J$-holomorphic disks with one marked point, Maslov index 2, and boundary lying on $T^k$ is transverse to $\rp^k\cap T^k $.
\item If $w:(D^2,\partial D^2)\to(\cp^k,T^k)$ is such a disk then $D_w\bar\partial_{J}$ is surjective.
\item $D_u\bar\partial_{J}$ is surjective for all $J$-holomorphic strips of Maslov index 1 or 2.
\item $\Phi_{\rp^k}(p)$+$\Phi_{T^k}(p)=0$ for all $p\in\rp^k\cap T^k$.
\end{enumerate}
Items 1-4 are true for all $k$.
For item 5, if $k\geq2$, then $\Phi_{\rp^k}(p)=0$ because the minimal Maslov number of $\rp^k$ is $k+1$.
Theorem 3.1 shows that $\Phi_{T^k}(p)=k+1$.
Thus $\Phi_{\rp^k}(p)+\Phi_{T^k}(p)=k+1$.
This is $0$ if and only if $k$ is odd.
(In the $k=1$ case, we actually have $\rp^1=T^1$, so $\Phi_{\rp^1}(p)+\Phi_{T^1}(p)=4=0$.)
Therefore, item 5 holds if and only if $k$ is odd.
\begin{prop}
$HF(\rp^{2n-1},T^{2n-1})$ is well-defined and the standard complex structure $J_0$ on $\cp^{2n-1}$ can be used to calculate it.
\end{prop}
\begin{proof}
(1) follows from Lemma 4.6 in \cite{oh2}.
(2), (3), and (4) will be proved in Section 5.
(5) was discussed prior to the proof.
\end{proof}

Henceforth we will use only $J=J_0$.

\section{Classification of Discs and Floer Trajectories}

The minimal Maslov number of $T^k$ is 2.
A consequence, as explained in Section 2, is that some of the boundaries of the 1-dimensional components of the compactified moduli space $\ccastripss{J}{\rp^k}{T^k}$ will consist of Maslov index 2 holomorphic disks with boundary lying on $T^k$.
For the Floer homology to be defined the number of such disks must be even.
The next theorem, due to Cho, classifies all holomorphic disks with boundary lying on $T^k$.
\begin{thm}[\cite{cho1} Theorem 10.1]
Let $w\colon (D^2,\partial D^2)\to(\mathbb CP^k,T^k)$ be a holomorphic map.
Then $w$ can be written as $w(z)=[w_0(z):\ldots:w_k(z)]$, where each $w_i$ is a finite Blaschke product.
That is,
$$w_i(z)=e^{\sqrt{-1}\theta_i}\prod_{j=1}^{\mu_i}{z-\alpha_{i,j}\over1-\bar\alpha_{i,j}z},$$
with $\theta_i\in\mathbb R$, $\mu_i\in\mathbb Z_{\geq0}$, $\alpha_{i,j}\in\mathrm{Int}(D^2)$, and $\cap_{i=0}^k\cup_{j=1}^{\mu_i}\alpha_{i,j}=\emptyset$.
Furthermore, the Maslov index of $w$ is
$$\mu(w)=2\sum_{i=0}^k\mu_i.$$
\end{thm}

Let $\mathcal M_1(T^k,2)$ denote the moduli space of holomorphic disks $w:(D^2,\partial D^2)\to(\cp^k,T^k)$ with Maslov index 2 and one marked point.
An immediate consequence of Theorem 3.1 is that the evaluation map
$$ev:\mathcal M_1(T^k,2)\to T^k$$
is a diffeomorphism.
Therefore,
$$\#(ev^{-1}(\rp^k\cap T^k))=\#(\rp^k\cap T^k)$$
and the number of Maslov index 2 disks is even if and only if $k$ is odd.

Theorem 3.1 allows us to determine all Floer trajectories as well.
Let $u$ be a trajectory, so $u$ satisfies (\ref{eq:dbar}) with $L_0=\rp^k$, $L_1=T^k$.
Using the Schwarz reflection principle, $u$ can be reflected about $\rp^k$ to obtain a map
$$\tilde u:\mathbb R\times[-1,1]\to \cp^k$$
with the properties
\begin{itemize}
\item $\tilde u|\mathbb R\times[0,1]=u$,
\item the energy of $\tilde u$ is twice that of $u$, and
\item both the top and bottom boundaries of $\tilde u$ lie on $T^k$.
\end{itemize}
$\mathbb R\times[0,1]$ is conformally equivalent to $D^2\setminus\{-1,1\}$, so the domain of $\tilde u$ can be thought of as $D^2\setminus\{-1,1\}$.
By the removable singularities theorem, $\tilde u$ extends to a holomorphic map $\tilde u:(D^2,\partial D^2)\to(\cp^k,T^k)$.
The Maslov index of $\tilde u$ (thinking of $\tilde u$ as a disc) is twice that of $u$.
Because $\tilde{u}(\mathbb R)\subset\rp^k$, Theorem 3.1 implies that $\tilde u$ is of the form
\begin{equation}\label{eq:class}
z\mapsto[\pm 1\prod_{i=0}^{j_0}\phi_{0,i}(z):\cdots:\pm1\prod_{i=0}^{j_k}\phi_{k,i}(z)]
\end{equation}
with $j_i\in\mathbb Z_{\geq 0}$, $\sum j_i=\mu(u)$, and $\phi_{l,i}\in\mathrm{Aut}(D^2\setminus\{-1,1\})$.
(This is straightforward to verify in case $\mu(u)=1,2$, which are the only cases we need.)

Therefore, identifying $\mathbb R\times[0,1]$ with $D^2\setminus\{-1,1\}$, we have shown that every holomorphic strip is the top half of a holomorphic disk of the form (\ref{eq:class}).
Note that under this identification of domains the $\mathbb R$ translation on $\mathbb R\times[0,1]$ corresponds to the action of $\mathrm{Aut}(D^2\setminus\{-1,1\})=\mathbb R$ on the top half of $D^2\setminus\{-1,1\}$.
In particular we have proved:

\begin{prop}
Let $p=[\epsilon_0:\cdots:\epsilon_k]\in\rp^k\cap T^k$.
There are exactly $k+1$ isolated holomorphic strips that start at $p$.
They are the top halves of the disks
$$w_0:z\mapsto[-\epsilon_0z:\epsilon_1:\cdots:\epsilon_k],$$
$$\vdots$$
$$w_k:z\mapsto[\epsilon_0:\cdots:\epsilon_{k-1}:-\epsilon_kz].$$
\end{prop}

\begin{prop}
Every holomorphic strip of Maslov index two is the upper half of a disk of the form
$$z\mapsto[\pm1:\cdots:\pm\phi_0(z):\cdots:\pm\phi_1(z):\cdots:\pm1],$$
with $\phi_0,\phi_1\in\mathrm{Aut}(D^2\setminus\{-1,1\})$.
($\phi_0(z)$ and $\phi_1(z)$ can both occur in the same homogeneous coordinate, in which case the coordinate is meant to be $\phi_0(z)\phi_1(z)$.)
\end{prop}

Using Propostion 3.2, we can now write down a formula for the boundary operator $\partial:CF(\rp^k,T^k)\to CF(\rp^k,T^k)$.
The formula is
\begin{equation}\label{eq:boundary}
\partial([\epsilon_0:\cdots:\epsilon_k])=\sum_{i=0}^k[\epsilon_0:\cdots:-\epsilon_i:\cdots:\epsilon_k].
\end{equation}
For example, if $k=3$, the basis of $CF(\rp^3,T^3)$ is ordered as
$$[1:1:1:1],[-1:1:1:1],[1:-1:1:1],[1:1:-1:1],[1:1:1:-1],$$
$$[-1:-1:1:1],[-1:1:-1:1],[-1:1:1:-1],$$
and the elements of $CF(\rp^3,T^3)$ are thought of as column vectors, then
\begin{equation}\label{eq:three}
\partial=
\left[
\begin{array}{cccccccc}
0 & 1 & 1 & 1 & 1 & 0 & 0 & 0 \\
1 & 0 & 0 & 0 & 0 & 1 & 1 & 1 \\
1 & 0 & 0 & 0 & 0 & 1 & 1 & 1 \\
1 & 0 & 0 & 0 & 0 & 1 & 1 & 1 \\
1 & 0 & 0 & 0 & 0 & 1 & 1 & 1 \\
0 & 1 & 1 & 1 & 1 & 0 & 0 & 0 \\
0 & 1 & 1 & 1 & 1 & 0 & 0 & 0 \\
0 & 1 & 1 & 1 & 1 & 0 & 0 & 0 \\
\end{array}
\right]
.
\end{equation}

\section{Calculation of $HF(\rp^{2n-1},T^{2n-1})$}
We are now ready to prove Theorem 1.1.
As mentioned before, $\rp^1$ and $T^1$ are Hamiltonian isotopic to each other in $\cp^1$, and it is known (see \cite{oh2}) that $HF(\rp^1,\rp^1)=(\z_2)^2$.
The next case is dimension 3.
\begin{lemma}
$$HF(\rp^3,T^3)=(\z_2)^4.$$
\end{lemma}
\begin{proof}
$\z_2$ is a field, so
$$\dim(HF)=\dim(\Ker(\partial))-\dim(\Ima(\partial))=\dim(CF)-2\dim(\Ima(\partial)).$$
From (\ref{eq:three}), we see that $\dim(\Ima(\partial))=2$.
Therefore, $\dim(HF)=2^3-4=4$.
\end{proof}

We prove the general case using induction.
Assume that
$$HF(\rp^{2n-1},T^{2n-1})=(\z_2)^{2^{n}},$$
and $n\geq 2$.
We then need to prove the result for $N=n+1$.
Let $CF(k)=CF(\rp^{2k-1},T^{2k-1})$, $\partial_k=\partial:CF(k)\to CF(k)$, and $HF(k)=HF(\rp^{2k-1},T^{2k-1})$.
Every point of $\rp^{2k-1}\cap T^{2k-1}$ can be written uniquely as
$$[1:\pm1:\pm1:\cdots:\pm1].$$
Let $(\pm1,\pm1,\cdots,\pm1)$ denote such a point.
Then a basis for $CF(k)$ is the set of all points $\sete{(\pm1,\cdots,\pm1)}$; elements of $CF(k)$ are formal sums of these points.

If $x\in CF(n)$ then $(1,1,x)$, $(-1,1,x)$, etc.\ will be used to denote elements in $CF(N)$.
For example, if 
$$x=(-1,1,1)\in CF(2)$$ 
then 
$$(1,1,x)=(1,1,-1,1,1);$$ 
if 
$$x=(1,1,1)+(-1,-1,1)$$ 
then 
$$(-1,1,x)=(-1,1,1,1,1)+(-1,1,-1,-1,1).$$

Let $\eta=\eta_k,\tilde\partial=\tilde\partial_k:CF(k)\to CF(k)$ be the maps
$$\eta(\epsilon_1,\cdots,\epsilon_k)=(-\epsilon_1,\cdots,-\epsilon_k),$$
$$\tilde\partial(\epsilon_1,\cdots,\epsilon_k)=\sum_{i=1}^k(\epsilon_1,\cdots,-\epsilon_i,\cdots,\epsilon_k).$$
From (\ref{eq:boundary}) it follows that $\partial_k=\tilde\partial_k+\eta_k$. 
Let $\pi:CF(N)\to CF(n)$ be the map that removes the first two coordinates, that is
$$\pi:(\epsilon_1,\epsilon_2,\epsilon_3,\cdots,\epsilon_N)\mapsto(\epsilon_3,\cdots,\epsilon_N).$$
$\pi$ is clearly surjective.
\begin{lemma}
$\partial_n\circ\pi=\pi\circ\partial_N.$
\end{lemma}
\begin{proof}
For $x\in CF(n)$, we calculate
$$\pi\circ\partial_N(\epsilon_1,\epsilon_2,x)=\pi\bigl((-\epsilon_1,\epsilon_2,x)+(\epsilon_1,-\epsilon_2,x)+(\epsilon_1,\epsilon_2,\tilde\partial_nx)+(-\epsilon_1,-\epsilon_2,\eta_nx)\bigr)=$$
$$x+x+\tilde\partial_nx+\eta_nx=\tilde\partial_nx+\eta_nx=\partial_nx=\partial_n\circ\pi(\epsilon_1,\epsilon_2,x).$$
\end{proof}

It follows that $\Ker(\partial_N)\subset \pi^{-1}(\Ker(\partial_n))$.
\begin{lemma}
Let $x\in\pi^{-1}(\Ker(\partial_n))$.
Then $x$ can be written uniquely as
\begin{equation}\label{eq:form}
\begin{array}{c}
x=\Bigl[(1,1,u)+(-1,-1,u)+(-1,1,v)+(1,-1,v)+\\
(1,1,w)+(-1,1,w)\Bigr]+(1,1,t)
\end{array}
\end{equation}
with $u,v,w\in CF(n)$, $t\in\Ker(\partial_n)$.
Moreover, $\partial_N(x)=0$ if and only if 
\begin{equation}\label{eq:relations}
\left\{
\begin{array}{l}
\partial_nv=w+t+\eta w,\\
\partial_nu=w+\eta t+\eta w,\\
\partial_nw=0.
\end{array}
\right.
\end{equation}
\end{lemma}
\begin{proof}
In (\ref{eq:form}), the term in brackets ranges over $\Ker(\pi)$ as $u,v,w$ range over $CF(n)$, and the final term maps (under $\pi$) onto all of $\Ker(\partial_n)$ as $t$ ranges over $\Ker(\partial_n)$.
It follows that $x\in\pi^{-1}(\Ker(\partial_n))$ can be written in the form (\ref{eq:form}).
The uniqueness of the expression follows from the fact that
$$\set{(1,1,u)+(-1,-1,u)}{u\in CF(n)}\cup\set{(-1,1,v)+(1,-1,v)}{v\in CF(n)}\cup$$
$$\set{(1,1,w)+(-1,1,w)}{w\in CF(n)}\cup\set{(1,1,t)}{t\in\Ker(\partial_n)}$$
is a linearly independent set.

If $\partial_N(x)=0$, then calculating $\partial_N$ of the right-hand side of (\ref{eq:form}), and setting equal to zero the sum of the remaining entries of all terms that start with the same two entries, yields the equations
\begin{equation}\label{eq:relationstwo}
\left\{
\begin{array}{l}
\tilde\partial w +w+\tilde\partial v+\eta v+t=0,\\
\eta w+w+\eta v+\tilde\partial v+t=0,\\
w+\tilde\partial w+\eta u+\tilde\partial u+\tilde \partial t=0,\\
w+\eta w+\tilde\partial u +\eta u+\eta t=0.
\end{array}
\right.
\end{equation}
For example, after applying $\partial_N$, the sum of the terms that start with $(-1,1)$ is (after cancellation)
$$(-1,1,\tilde\partial w)+(-1,1,w)+(-1,1,\tilde\partial v)+(-1,1,\eta v)+(-1,1,t).$$
Since this sum must be 0, it follows that
$$\tilde\partial w+w+\tilde\partial v+\eta v+t=0.$$
The remaining equations come from examining the terms that start with $(1,-1)$, $(1,1)$, and $(-1,-1)$, respectively.

Using the fact that $\partial=\tilde\partial+\eta$ and $t\in\Ker(\partial)$, it is straightforward to check that (\ref{eq:relations}) is equivalent to (\ref{eq:relationstwo}).
Indeed, labeling the equations in (\ref{eq:relations}) as (\ref{eq:relations}.1), (\ref{eq:relations}.2), and (\ref{eq:relations}.3) and the equations in (\ref{eq:relationstwo}) as (\ref{eq:relationstwo}.1), (\ref{eq:relationstwo}.2), (\ref{eq:relationstwo}.3), and (\ref{eq:relationstwo}.4), we have
\begin{displaymath}
\left\{
\begin{array}{ccc}
(\ref{eq:relations}.1),(\ref{eq:relations}.3) & \Rightarrow & (\ref{eq:relationstwo}.1)\\
(\ref{eq:relations}.1),(\ref{eq:relations}.2) & \Rightarrow & (\ref{eq:relationstwo}.2)\\
(\ref{eq:relations}.2) & \Rightarrow & (\ref{eq:relationstwo}.3)\\
(\ref{eq:relations}.2),(\ref{eq:relations}.3) & \Rightarrow & (\ref{eq:relationstwo}.4)\\
\end{array}
\right.
\end{displaymath}
and
\begin{displaymath}
\left\{
\begin{array}{ccc}
(\ref{eq:relationstwo}.2) & \Rightarrow & (\ref{eq:relations}.1)\\
(\ref{eq:relationstwo}.4) & \Rightarrow & (\ref{eq:relations}.2)\\
(\ref{eq:relationstwo}.3),(\ref{eq:relationstwo}.4) & \Rightarrow & (\ref{eq:relations}.3).\\
\end{array}
\right.
\end{displaymath}
\end{proof}

Let us denote $x$ of the form (\ref{eq:form}) as $x(u,v,w,t)$.
Consider the map
$$\alpha:CF(n)\oplus CF(n)\oplus\Ker(\partial_n)\oplus\Ker(\partial_n)\to\Ima(\partial_n)\oplus\Ker(\partial_n),$$
$$(u,v,w,t)\mapsto(\partial\eta u+\partial v,\partial\eta u +w+t+\eta w).$$
Taking $u=w=0$ shows that $\alpha$ is onto.
\begin{lemma}
$$\Ker(\alpha)=\set{(u,v,w,t)}{u,v,w,t\textrm{ satisfy (\ref{eq:relations})}}.$$
\end{lemma}
\begin{proof}
Since $w$ is in the domain of $\alpha$, we have $\partial w=0$.
Furthermore, $\alpha(u,v,w,t)=0$ if and only if
$$\partial\eta u+\partial v=0,$$
$$\partial\eta u+w+t+\eta w=0.$$
Since $\partial\eta=\eta\partial$ and $\eta^2=Id$, the second equation is equivalent to $\partial u=w+\eta t+\eta w$.
The first equation then becomes $\partial v=\eta\partial u=w+t+\eta w$.
These are precisely the equations given in (\ref{eq:relations}).
\end{proof}

It follows that there is a bijection between $\Ker(\alpha)$ and $\Ker(\partial_N)$.
The bijection is the obvious one: $(u,v,w,t)\mapsto x(u,v,w,t)$.

We now complete the proof.
Because $\alpha$ is onto, we have
$$\dim(\Ker(\alpha))=2\dim(CF(n))+2\dim(\Ker(\partial_n))-\dim(\Ima(\partial_n))-\dim(\Ker(\partial_n))$$
$$=2\dim(CF(n))+\dim(HF(n))={1\over 2}\dim(CF(N))+\dim(HF(n)).$$
Because $\Ker(\alpha)$ and $\Ker(\partial_N)$ have the same cardinality, we have
$$\dim(\Ker(\partial_N))={1\over 2}\dim(CF(N))+\dim(HF(n)).$$
It follows that
$$\dim(HF(N))=2\dim(HF(n))=2\cdot 2^n=2^N.$$
This completes the proof of Theorem 1.1.

\section{Verification of Regularity}
In this section we prove items 2, 3, and 4 from the end of Section 2.
We start with 3:
\begin{lemma}
Let $w:(D^2,\partial D^2)\to(\cp^n,T^n)$ be a holomorphic disc with Maslov index 2.
Then $D_w\bar\partial$ is surjective.
\end{lemma}
\begin{proof}
This is proved in \cite{cho1} Theorem 10.2 and \cite{choo} Theorem 6.1.
The idea is that $\bar\partial$ splits as a direct sum of one-dimensional operators with non-negative Maslov index.
Such one-dimensional operators are surjective; see for example \cite{ms} Theorem C.1.10.
\end{proof}

Recall that $\mathcal M_1(D^2,2)$ denotes the moduli space of the discs in Lemma 5.1 with one marked point.
Lemma 5.1 implies that $\mathcal M_1(D^2,2)$ is a smooth manifold, and the discussion following Theorem 3.1 then shows that
$$ev:\mathcal M_1(D^2,2)\to T^n$$
is a diffeomorphism.
This proves item 2.

Finally we prove item 4.
The idea is similar to Lemma 5.1: The linearized operator splits as a direct sum of 1-dimensional operators with nonnegative Maslov index, and these are always surjective.
We first prove this last statement.
\begin{lemma}
Let
$$\bar\partial:W^{1,p}_\lambda(\mathbb R\times[0,1],\mathbb C)\to L^p(\Omega^{0,1}(\mathbb R\times[0,1])\otimes\mathbb C))$$
be the standard Cauchy-Riemann operator.
Suppose $\lambda:\mathbb R\times\{0,1\}\to\Lambda(\mathbb C^n)$ has nonnegative Maslov index, and $\lambda(\pm\infty,0)=\mathbb R$, $\lambda(\pm\infty,1)=i\cdot\mathbb R$.
Then $\bar\partial$ is surjective.
\end{lemma}
\begin{proof}
Let $\mu\geq0$ be the Maslov index.
Suppose $\eta\in L^p(\mathbb R\times[0,1],\mathbb C)$ is in the $L^2$ orthogonal complement of the image of $\bar\partial$.
Then a straightforward integration by parts argument shows that $\partial \eta=0$ and $\eta$ satisfies the same boundary conditions $\lambda$.
Using the Schwarz reflection principle, define $\tilde\eta:\mathbb R\times[-1,1]\to\mathbb C$ that satisfies
\begin{itemize}
\item $\tilde\eta|\mathbb R\times[0,1]=\eta$,
\item the $W^{1,2}$ norm of $\tilde\eta$ is twice that of $\eta$, and
\item $\tilde\eta(s,1)\in\lambda(s,1)$, $\tilde\eta(s,-1)\in\overline{\lambda(s,1)}$.
\end{itemize}

$\mathbb R\times[-1,1]$ is conformally equivalent to $D^2\setminus\{-1,1\}$ so the domain of $\tilde\eta$ can be viewed as $D^2\setminus\{-1,1\}$.
The Lagrangian boundary condition on the strip extends continuously to a Lagrangian boundary condition on all of $\partial D^2$.
Moreover, it has Maslov index $2\mu$.
Therefore, $\tilde\eta$ extends smoothly to all of $D^2$ by the removable singularities theorem.
Again using integration by parts, $\tilde\eta$ can be viewed as an element in the kernel of the adjoint of a Riemann-Hilbert problem on the disc with Maslov index $2\mu-1\geq-1$ (see for example \cite{oh4} formula (5.8)).
But the cokernel of such a Riemann-Hilbert problem is $0$, hence $\tilde\eta=0$ and thus $\eta=0$.
This shows that $\bar\partial$ is surjective.
\end{proof}

\begin{lemma}
Let $u$ be a holomorphic strip with $\mu(u)=1$ or $2$.
Then $D_u\bar\partial$ is surjective.
\end{lemma}
\begin{proof}
Consider first the case $\mu(u)=1$.
Using Proposition 3.2, we can choose coordinates so that $u$ is the upper half of the disc
$$(D^2,\partial D^2)\to(\cp^n,T^n),\qquad z\mapsto(z,1,\cdots,1).$$
The linearization $D_u\bar\partial$ splits as a direct sum 
$$D_u\bar\partial=\bar\partial_1\oplus\bar\partial_0\oplus\cdots\oplus\bar\partial_0$$
where
$$\bar\partial_i:W^{l,p}_{\lambda_i}(\mathbb R\times[0,1],\mathbb C)\to W^{l-1,p}(\mathbb R\times[0,1],\mathbb C).$$
The $\lambda_i$ subscript denotes a Lagrangian boundary condition with Maslov index $i$.
By the previous lemma, each $\bar\partial_i$ is surjective, and thus $D_u\bar\partial$ is surjective.

The case where $\mu(u)=2$ is similar.
The difference is that the linearized operator splits as
$$\bar\partial_1\oplus\bar\partial_1\oplus\bar\partial_0\oplus\cdots\oplus\bar\partial_0$$
or 
$$\bar\partial_2\oplus\bar\partial_0\oplus\cdots\oplus\bar\partial_0.$$
$\bar\partial_2$ denotes the operator with boundary conditions of Maslov index 2.
Again the previous lemma then implies that $D_u\bar\partial$ is surjective.
\end{proof}

\section{Novikov Ring Coefficients}
In this section we compute the Floer homology with coefficients in the universal Novikov ring.
The universal Novikov ring is 
$$\unov=\Bigl\{\,\,{\sum_{(\lambda,n)\in\mathbb R\times\z} b_{(\lambda,n)}T^{\lambda}e^{n}}\,\,|\,\,{b_{(\lambda,n)}\in\z_2\textrm{ and }\forall C\in\mathbb R,\#\set{b_{(\lambda,n)}\neq 0}{\lambda<C}<\infty\,\Bigr\}}.$$
See \cite{fooo} Section 20 for more details.
The result is
\begin{thm}
$$HF(\rp^{2n-1},T^{2n-1}:\Lambda_{\z_2})=(\Lambda_{\z_2})^{2^{n}}.$$
\end{thm}

We first explain what is meant by homology with coefficients in $\unov$.
We start with the following definition, taken from \cite{u}.
\begin{dfn}
A graded filtered Floer-Novikov complex $\mathfrak c$ consists of the following data:
\begin{enumerate}
\item A principal $\Gamma$-bundle $P\to S$, where $S$ is a finite set and $\Gamma$ is a finitely generated abelian group.
\item An action functional ${\mathcal A}:P\to \mathbb R$ and a period homomorphism $\omega:\Gamma \to \mathbb R$ satisfying ${\mathcal A}(g\cdot p)={\mathcal A}(p)-\omega(g)$.
\item A grading $gr:P\to \mathbb Z$ and a degree homomorphism $d:\Gamma \to \mathbb Z$ satisfying $gr(g\cdot p)=gr(p)+d(g)$.
\item A map $n':P\times P\to R$ ($R$ is a commutative ring) satisfying the following conditions:
\begin{enumerate}
\item $n'(p,p')=0$ unless ${\mathcal{A}}(p)>{\mathcal{ A}}(p')$ and $gr(p')=gr(p)-1$,
\item $n'(g\cdot p,g\cdot p')=n'(p,p')$,
\item for each $p\in P$, the formal sum 
$$\partial' p=\sum_{q\in P}n'(p,q)q$$ 
belongs to the Floer chain complex 
$$C_{*}(\mathfrak c)=\{\,\sum_{q\in P}a_qq\,|\,a_q\in R,\#\{\,q\,|\,a_q\neq 0,\, {\mathcal {A}}(q)>C\,\}<\infty \,\forall\, C\in \mathbb R\,\},$$
\item and with the Novikov ring of $\Gamma$  defined by
$$\Lambda_{\Gamma,\omega}=\{\,\sum_{g\in\Gamma}b_gg\,|\,b_g\in R,\,\#\{\,g\,|\,b_q\neq 0,\,\omega(g)<C\,\}<\infty\, \forall\, C\in \mathbb R\,\},$$
we require that $C_*$ inherits the structure of a $\Lambda_{\Gamma,\omega}$-module, the operator $\partial':P\to C_*$ extends to a $\Lambda_{\Gamma,\omega}$-module homomorphism $\partial':C_*\to C_*$, and it satisfies $\partial'^2=0$.
\end{enumerate}
\end{enumerate}
\end{dfn}
(The notation $n'$ and $\partial'$ is used because we will need to refer to $n$ and $\partial$ as defined in Section 2.)

We will construct a chain complex $CF(\rp^k,T^k:\nov)$ (for $k=2n-1$) that plays the role of $C_*(\mathfrak c)$ in the definition.
Then $HF(\rp^k,T^k:\nov)$ is defined to be the homology of $C_*(\mathfrak c)=CF(\rp^k,T^k:\nov)$.
There is a homomorphism $\nov\to\unov$ given by
$$\sum_{g\in\Gamma}b_gg\mapsto\sum_{g\in\Gamma}b_gT^{\omega(g)}e^{d(g)/2}.$$
Using this homomorphism we can define a $\unov$ chain complex 
$$\partial'\otimes1:C_*(\mathfrak c)\otimes_{\nov}\unov\to C_*(\mathfrak c)\otimes_{\nov}\unov.$$
By definition, $HF(\rp^k,T^k:\unov)$ is the homology of this complex.
To compute it we will first calculate $HF(\rp^k,T^k:\nov)$ and then use the fact that $\nov$ is a field to conclude that $HF(\rp^k,T^k:\unov)=HF(\rp^k,T^k:\nov)\otimes_{\nov}\unov$.

We turn to constructing $CF(\rp^k,T^k:\nov)$.
Let $q_0=[1:\cdots:1]\in\rp^k\cap T^k$.
We can also think of $q_0$ as the constant path $q_0:[0,1]\to\cp^k$, $q_0(t)=q_0$ for all $t$.
Likewise, any point of $\rp^k\cap T^k$ can be thought of as a constant path.
In the following, it should be clear from context when we are taking this point of view.
Let 
$$\Omega=\set{\gamma:[0,1]\to\cp^k}{\gamma(0)\in\rp^k,\gamma(1)\in T^k},$$
and let $\Omega(q_0)$ be the path component of $\Omega$ containing the constant path $q_0$.
Any $[u]\in\pi_1(\Omega(p))$ is a map $u:S^1\times[0,1]\to\cp^k$ (well-defined up to homotopy).
Let
$$I_\omega:\pi_1(\Omega(q_0))\to\mathbb R$$
be the homomorphism given by $I_\omega([u])=\int u^*\omega$.
Notice that $u$ defines a bundle pair
$$(u^*T\cp^k,(u|S^1\times\{0\})^*T\rp^k\amalg(u|S^1\times\{1\})^*TT^k)$$ 
over the cylinder $S^1\times[0,1]$.
Define another homomorphism
$$I_\mu:\pi_1(\Omega(q_0))\to\z$$
by letting $I_\mu([u])$ be the Maslov index of this bundle pair.

Since $N=\Ker(I_\mu)\cap\Ker(I_\omega)$ is a normal subgroup of $\pi_1(\Omega(q_0))$, it defines a normal cover $\widetilde\Omega$ of $\Omega(q_0)$.
Explicitly, points of $\widetilde\Omega$ are equivalence classes of pairs $(\gamma,u)$, where $\gamma\in\Omega(q_0)$ and $u$ is a path from $q_0$ to $\gamma$.
$[\gamma,u]=[\gamma',u']$ if and only if $\gamma=\gamma'$, $I_\mu(u)=I_\mu(u')$, and $I_\omega(u)=I_\omega(u')$.
The automorphism group of the cover is $\pi_1(\Omega(q_0))/N$.
Let $\Gamma$ be this group.

\begin{lemma}
$\Gamma$ is isomorphic to $\mathbb Z$.
\end{lemma}
\begin{proof}
By construction, $\Gamma$ is isomorphic to $\pi_1(\Omega(q_0))/\Ker(I_\mu)\cap\Ker(I_\omega)$.
The latter group is isomorphic to $\Ima(I_\mu)\oplus\Ima(I_\omega)\subset\z\oplus\mathbb R$.
$\rp^k$ and $T^k$ are both monotone and $\pi_1(\rp^k)$ is torsion, so Proposition 2.7 in \cite{oh1} implies that there exists a constant $c>0$ such that $I_\omega=cI_\mu$.
Thus $\Ima(I_\mu)\oplus\Ima(I_\omega)$ is isomorphic to $\z$.
\end{proof}

$\rp^k$ and $T^k$ are both orientable (because $k$ is odd), so the image of $I_\mu$ is contained in $2\z$.
Moreover, if $u_0$ denotes the unique holomorphic strip of Maslov index 1 from $q_0$ to $[-1:1:\cdots:1]$, and $u_0'$ denotes the unique holomorphic strip of Maslov index 1 from $[-1:1:\cdots:1]$ to $q_0$, then $I_\mu(u_0\#u_0')=2$.
Thus the image of $I_\mu\oplus I_\omega$ is generated by $(2,2c)$, where $c$ is the energy of $u_0$ (which is the same as the energy of any strip with Maslov index 1).
We denote the preimage of this element in $\Gamma$ as $e$, and we use multiplicative notation to describe $\Gamma$.
That is, 
$$\Gamma=\set{e^j}{j\in\z},\qquad e^j\cdot e^l=e^{j+l}.$$
We take the degree and period homomorphisms to be 
$$\omega(e^j)=2jc,$$
$$d(e^j)=-2j.$$

Let
$$P=\set{[\gamma,u]\in\tilde\Omega}{\gamma\textrm{ is a constant path}}$$
and
$$S=\rp^k\cap T^k.$$
There is an obvious projection $P\to S$ that makes $P$ a principal $\Gamma$-bundle.
We let $\mathcal A$ be the action functional
$$\mathcal A:P\to\mathbb R,\qquad[\gamma,u]\mapsto-\int u^*\omega.$$
Next we define the grading $gr:P\to\z$.
If $[q,u]\in P$, then $u$ is a map $u:[0,1]\times[0,1]\to\cp^n$, with $u(\cdot,1)\in T^n$, $u(\cdot,0)\in\rp^n$, $u(0,t)=q_0$, and $u(1,t)=q$.
Extend the domain of $u$ to $\mathbb R\times[0,1]$, by keeping the value of $u$ fixed on each component in the complement of $(0,1)\times[0,1]$.
Then we define 
$$gr([q,u])=-\mu(u).$$
For example, if $u$ is homotopic to an isolated Floer trajectory connecting $q_0$ to $q$, then $gr([q,u])=-1$.

Finally, we take $R=\z_2$ and define $n':P\times P\to \z_2$ by letting $n'([q_1,u],[q_2,u'])$ be the number mod-2 of holomorphic strips $w$ of Maslov index 1 that start at $q_1$, end at $q_2$, and are in the homotopy class of $u^{-1}\#u'$ (that is, as paths in $\Omega(q_0)$, $u\#w$ and $u'$ are homotopic).

The Novikov ring $\nov$ from Definition 6.2 can now be described as
$$\nov=\Bigl\{\,\,\sum_{k\in\z}a_ke^k\,\,|\,\,a_k\in\z_2\textrm{, there exists $k_0$ such that $a_k=0$ for all $k\leq k_0$}\,\Bigr\}.$$
Note that $\nov$ is actually a field.

We now calculate the homology.
In order to describe the matrix for $\partial'$ we need to first choose an ordered basis for $CF(\rp^k,T^k:\nov)$.
We do this as follows: 
For each $q\in\rp^k\cap T^k$, let $u_q$ be any holomorphic strip from $q_0$ to $q$.
Such a holomorphic strip exists by the results of Section 3, but of course it does not necessarily have to be isolated.
Recall that $c$ is the energy of an isolated strip (that is, a strip of Maslov index 1).
If the number of 1's in the homogeneous coordinates of $q$ is even then the energy of $u_q$ is $2jc$ for some $j$.
Let $v_q=e^{-j}\cdot [q,u_q]\in P$.
Then by construction we have $\mathcal A(v_q)=0$ and $gr(v_q)=d(e^{-j})+gr([q,u_q])=2j-2j=0$.
If the number of $1$'s is odd then the energy is $(2j+1)c$.
Let $v_q=e^{-j}\cdot [q,u_q]\in P$ in this case, and then we have $\mathcal A(v_q)=-c$ and $gr(v_q)=-1$.
Thus
\begin{displaymath}
gr(v_q)=\left\{
\begin{array}{ll}
0 & \textrm{$q$ has an odd number of 1's}\\
-1 & \textrm{$q$ has an even number of 1's,}\\
\end{array}
\right.
\end{displaymath}
and $\set{v_q}{q\in\rp^k\cap T^k}$ is a basis for $CF(\rp^k,T^k:\nov)$ over $\nov$.
Order this basis in any way with $v_{q_0}$ first and denote it as $(v_{q_0},v_{q_1},\cdots,v_{q_{2^k}})$.
Let $u_{q_i}'$ be such that $v_{q_i}=[q_i,u_{q_i}']$.
$\partial'$ can now be described by a matrix with respect to this basis.

To calculate the dimension of the image of $\partial'$ we compare the matrix to the matrix for $\partial$.
The ordered basis of $CF(\rp^k,T^k:\nov)$ that we constructed induces an ordered basis of $CF(\rp^k,T^k)$ in an obvious way (they both have bases in bijective correspondence to $\rp^k\cap T^k$).
We think of $\partial$ as being a matrix with respect to this basis.

\begin{lemma}
The columns of $\partial$ corresponding to the $v_q$'s with $q$ having an even number of $1$'s in the homogeneous coordinates agree with the corresponding columns in $\partial'$.
The remaining columns of $\partial'$ are $e$ times the corresponding columns in $\partial$.
\end{lemma}
\begin{proof}
Consider the first statement.
Let $q_i$ be a point of $\rp^{k}\cap T^{k}$ with an even number of homogeneous coordinates equal to $1$, so $gr(v_{q_i})=0$.
If $q_j$ is a point with an isolated Floer trajectory from $q_i$ to $q_j$, then the number of homogeneous coordinates of $q_j$ equal to $1$ must be odd.
Therefore $gr(v_{q_j})=-1$.
Let $u_{ij}$ denote the unique isolated Floer trajectory from $q_i$ to $q_j$.
Then 
$$gr([q_i,u_{q_i}'\#u_{ij}])=0-1=gr([q_j,u_{q_j}'])$$ 
and 
$$\mathcal{A}([q_i,u_{q_i}'\#u_{ij}])=0-c=\mathcal{A}([q_j,u_{q_j}']).$$
Therefore $[q_j,u_{q_i}'\#u_{ij}]=[q_j,u_{q_j}']$.
Moreover, by the grading of $v_{q_i}$ and $v_{q_j}$, any holomorphic strip $u$ from $q_i$ to $q_j$ that satisfies $[q_i,u_{q_i}'\#u]=[q_j,u_{q_j}']$ must have Maslov index 1, and hence by the uniqueness of $u_{ij}$ it follows that $u=u_{ij}$.
Thus 
$$n'(v_{q_i},v_{q_j})=n(q_i,q_j)=1.$$
If $q_j$ is a point with no isolated Floer trajectories from $q_i$ to $q_j$ then, by definition, $n'(v_{q_i},v_{q_j})=0=n(q_i,q_j)$.
Therefore we have
\begin{displaymath}
\left[
\begin{array}{c}
n'(v_{q_i},v_{q_0})\\
\vdots\\
n'(v_{q_i},v_{q_{2n}})\\
\end{array}
\right]
=
\left[
\begin{array}{c}
n(q_i,q_0)\\
\vdots\\
n(q_i,q_{2n})\\
\end{array}
\right].
\end{displaymath}
That is, the column of $\partial'$ corresponding to $v_{q_i}$ is the same as the column of $\partial$ corresponding to $q_i$.

Now consider the second statement.
Let $q_i$ be a point of with an odd number of homogeneous coordinates equal to $1$, so $gr(v_{q_i})=-1$.
If $q_j$ is a point with an isolated Floer trajectory from $q_i$ to $q_j$, then the number of homogeneous coordinates of $q_j$ equal to 1 must be even.
Therefore $gr(v_{q_j})=0$, and hence $gr(e\cdot v_{q_j})=-2$.
Let $u_{ij}$ denote the unique isolated Floer trajectory from $q_i$ to $q_j$.
Reasoning the same way as above, we have $[q_j,u_{q_i}'\#u_{ij}]=e\cdot [q_j,u_{q_j}']$ and thus $n'(v_{q_i},v_{q_j})=e$.
Points $q_j$ with no connecting isolated Floer trajectory have $n'(v_{q_i},v_{q_j})=0$.
Thus the column for $\partial'$ is $e$ times the column for $\partial$.
\end{proof}

\begin{cor}
$$HF(\rp^{2n-1},T^{2n-1}:\nov)=(\nov)^{2^n}.$$
\end{cor}
\begin{proof}
$\nov$ is a field and includes $\z_2=\z_2\cdot e^0\subset \nov$ as a subfield.
Therefore the previous lemma implies that there is an invertible diagonal matrix $M$ such that, under the identification
$$CF(\rp^n,T^n)\otimes_{\z_2}\nov=CF(\rp^n,T^n:\nov)$$
given by the chosen basis, we have $\partial\otimes1=M\circ\partial'$.
The dimension of the image of $M\circ\partial'$ is the same as $\partial'$ because $M$ is invertible.
The dimension of the image of $\partial\otimes1$ over $\nov$ is the same as the dimension of the image of $\partial$ over $\z_2$.
The result now follows from Theorem 1.1.
\end{proof}

Recall that $HF(\rp^{2n-1},T^{2n-1}:\unov)$ is defined to be the homology of the complex
$$\partial'\otimes1:C_*(\mathfrak c)\otimes_{\nov}\unov\to C_*(\mathfrak c)\otimes_{\nov}\unov$$
with $C_*(\mathfrak c)=CF(\rp^{2n-1},T^{2n-1}:\nov)$.
Therefore the previous corollary and the fact that $\nov$ is a field implies that 
$$HF(\rp^{2n-1},T^{2n-1}:\unov)=(\unov)^{2^n},$$
and Theorem 6.1 is proved.

\section{Volume Minimization of $T^{2n-1}$}
In this section we briefly discuss the problem of minimizing the volume of $\phi(T^k)$ for $\phi$ a Hamiltonian diffeomorphism.
In \cite{oh3} it is proved that $T^k$ is locally volume minimizing, that is it has minimal volume among all small Hamiltonian deformations.
However, it is unknown if it is globally volume minimizing.

One approach to gain information about this problem is to use the following Crofton type formula.
\begin{thm}[\cite{van} Proposition 2.10]
Let $L$ be a Lagrangian submanifold in $\cp^k$.
Then 
$$vol(L)=c_k\int_{U(k+1)/O(k+1)}\#(L\cap g\cdot \rp^k)dg,$$
where $c_k$ is a constant that does not depend upon $L$.
\end{thm}
Generically $g\cdot\rp^k$ and $\rp^k$ intersect in $k+1$ points.
Therefore the theorem implies that for any Lagrangian $L$ we have
\begin{displaymath}
\frac{\textrm{vol}(L)}{\textrm{vol}{(\rp^k)}}\geq \frac{\mathrm{min}_g\#(L\cap g\cdot\rp^k)}{k+1}.
\end{displaymath}
Theorem 1.2 implies that $\#(\phi(T^{2n-1})\cap g\cdot\rp^{2n-1})\geq 2^n$, and therefore
\begin{displaymath}
\frac{\textrm{vol}(\phi(T^{2n-1})}{\textrm{vol}{(\rp^{2n-1}})}\geq \frac{2^n}{2n}.
\end{displaymath}
This proves Corollary 1.3

We want to compare this estimate to the volumes of $\rp^{2n-1}$ and $T^{2n-1}$.
$S^1$ acts on the unit sphere $S^{4n-1}\subset \mathbb C^{2n}$ and
$$\mathbb CP^{2n-1}=S^{4n-1}/S^1.$$
The symplectic form on $\cp^{2n-1}$ is the one obtained by symplectic reduction from the standard symplectic form on $\mathbb C^{2n}$.
Moreover, the projection $S^{4n-1}\to\mathbb CP^{2n-1}$ is an isometric submersion, that is it is an isometry on the orthogonal complement of the kernel of $TS^{4n-1}\to T\mathbb CP^{2n-1}$.
Using this fact, it follows that the volume of $\rp^{2n-1}$ is half the volume of the unit sphere $S^{2n-1}$.
Thus
$$\mathrm{vol}(\rp^{2n-1})=\frac{\pi^n}{(n-1)!}.$$
Moreover, the Clifford torus is the quotient of the torus $(S^1({1/\sqrt{2n}}))^{2n}\subset S^{4n-1}$ by $S^1$.
Thus
$$\mathrm{vol}(T^{2n-1})=\frac{1}{2\pi}\Biggl(\frac{2\pi}{\sqrt{2n}}\Biggr)^{2n}.$$
Therefore
$$\frac{\mathrm{vol}(T^{2n-1})}{\mathrm{vol}{(\rp^{2n-1})}}=\frac{(2\pi)^{n-1}(n-1)!}{n^n}.$$

For $n=1,2$ we thus get:
\begin{displaymath}
\begin{array}{ccccccc}
\frac{\textrm{vol}(T^1)}{\textrm{vol}(\rp^1)} &
= &
1
&
,
&
\frac{\textrm{vol}(\phi(T^1))}{\textrm{vol}(\rp^1)} &
\geq 
&
1, \\
\\
\frac{\textrm{vol}(T^3)}{\textrm{vol}(\rp^3)} &
= &
\frac{\pi}{2}\approx 1.57
&
,
&
\frac{\textrm{vol}(\phi(T^3))}{\textrm{vol}(\rp^3)} &
\geq 
&
1. 
\end{array}
\end{displaymath}
We recover the well-known fact that $S^1$ is volume minimizing in $\cp^1$, but for higher dimensions the comparison leaves the problem unanswered.


\end{document}